\documentclass[11pt,twoside]{article}
\usepackage[utf8]{inputenc}
\usepackage[T1]{fontenc}
\usepackage[margin=1in]{geometry}
\usepackage{amsmath,amsthm,amssymb}
\usepackage{newcent,microtype}
\usepackage[colorlinks=true
						,linkcolor=blue
						,citecolor=blue
						]{hyperref}
\usepackage{mathtools}
\usepackage{tikz}
\usepackage{xfrac}
\usepackage{array}

\swapnumbers
\newtheorem{theorem}{Theorem}[section]
\newtheorem{lemma}[theorem]{Lemma}

\newtheorem{definition}[theorem]{Definition}
\newtheorem{hypothesis}[theorem]{Hypothesis}
\newtheorem*{introthm}{Theorem}

\newcommand\cf{{\em c.f.~}}

\newcommand\ie{{\em i.e., }}

\DeclareMathOperator\ad{ad}

\DeclareMathOperator\Aut{Aut}

\DeclareMathOperator\End{End}

\DeclareMathOperator\id{id}

\DeclareMathOperator\Spec{Spec}

\DeclareMathOperator\Sym{Sym}
\DeclareMathOperator\Z{Z}

\newcommand\la{\langle}

\newcommand\ra{\rangle}

\newcommand\size[1]{\lvert #1\rvert}

\newcommand\CC{\mathbb{C}}
\newcommand\FF{\mathbb{F}}

\newcommand\NN{\mathbb{N}}

\newcommand\ZZ{\mathbb{Z}}

\newcommand\op{\mathrm}
\renewcommand\cal{\mathcal}
\renewcommand\frak{\mathfrak}

\newcommand\al{\alpha}

\newcommand\ep{\varepsilon}

\newcommand\lm{\lambda}

\newcommand\cc{\op{cc}}

\newcommand\rt[1]{{\rm #1}}
\newcommand\affrt[1]{\hat{\rt #1}}

\newcommand\sh[1]{(\mathrm #1)}

\makeatletter
\renewcommand*{\ext@figure}{lot}
\let\c@figure\c@table
\let\ftype@figure\ftype@table

\makeatother

\allowdisplaybreaks
\begin{document}
\abovedisplayskip=0.5em
\belowdisplayskip=0.5em

\title{Linear idempotents in Matsuo algebras}
\author{F. Rehren}
\date{5th February 2015}

\maketitle

\begin{abstract}
  \noindent
  Matsuo algebras are an algebraic incarnation of 3-transposition groups with a parameter $\al$,
  where idempotents takes the role of the transpositions.
  We show that a large class of idempotents in Matsuo algebras
  satisfy the Seress property,
  making these nonassociative algebras well-behaved analogously to associative algebras, Jordan algebras and vertex (operator) algebras.
  We calculate eigenvalues in the Matsuo algebra of $\Sym(n)$ for any $\al$,
  generalising some vertex algebra results for which $\al=\frac{1}{4}$.
  Finally, in the Matsuo algebra of the root system ${\rm D}_n$,
  we show $n-3$ conjugacy classes of involutions coming from the Weyl group
  are in natural bijection with idempotents in the algebra via their fusion rules.
\end{abstract}

  Idempotents play a distinguished role in algebras.
  In matrix algebras and generally in associative algebras,
  idempotents are projections onto subspaces,
  with eigenvalues $1$ and $0$.
  In nonassociative algebras, the situation is more subtle.
  In for example the classical theory of Jordan algebras,
  structural theorems depend on the existence of idempotents,
  which now admit eigenvalues $1$, $0$ and $\frac{1}{2}$;
  a key result is that the product of a $\phi$-eigenvector with a $\psi$-eigenvector
  is a sum of $\phi\star\psi$-eigenvectors,
  according to the fusion rules $\Phi(\al)$ of Table~\ref{tbl jordan}
  with $\al = \frac{1}{2}$.
  In some ways, idempotents are also analogous
  to $\frak{sl}_2$-subalgebras of Lie algebras.
  
  These ideas are captured by {\em axial algebras},
  which are algebras generated by idempotents satisfying some fusion rules $\Phi$.
  An important source of axial algebras
  are the weight-$2$ subalgebras of a special class of vertex (operator) algebras,
  where the fusion rules come from the representation theory of the Virasoro algebra.
  The most famous instance is the Griess algebra---%
  the weight-$2$ subalgebra of the vertex algebra $V^\natural$---%
  whose automorphism group is the Monster.
  This algebra is generated by idempotents
  with eigenvalues $1,0,\frac{1}{4},\frac{1}{32}$
  satisfying the Ising fusion rules.
  These fusion rules are $\ZZ/2$-graded,
  so each such idempotent induces an involution,
  in the conjugacy class $\sh{2A}$.
  These involutions generate the entire group and have pairwise products of order at most $6$,
  whence the Monster is a $6$-transposition group.
  
  The fusion rules $\Phi(\al)$ are simpler but also $\ZZ/2$-graded.
  Using the grading, in \cite{hrs} it is shown that the involutions induced from $\Phi(\al)$-idempotents
  generate a $3$-transposition group $G$
  if and only if those idempotents generate a {\em Matsuo algebr}a $M_\al(\cal G)$,
  and this is always the case when $\al\neq0,\frac{1}{2}$%
  \footnote{
    when $\al=\frac{1}{2}$,
    the Jordan algebras with associated $3$-transposition groups
    are classified in \cite{jordan3trgps}}
  or $1$.
  Here $\cal G$ is the Fischer space, a graph, of $G$.
  
  In this text we investigate further algebraic properties of such $M_\al(\cal G)$.
  For a $3$-transposition group $G$,
  its Matsuo algebra $M_\al(\cal G)$ may be thought of as an alternative to its ordinary group algebra $\FF G$;
  the theory of group algebras, for example relating central primitive idempotents to irreducible characters,
  has already been fruitfully well-developed \cite{passman}.
  The key to our approach
  is that {\em linear} idempotents provide a direct link
  between group-theoretic properties of $G$
  and structural properties of the algebra $M_\al(\cal G)$.
  Our results, outlined below, are a first step of the same programme for arbitrary idempotents in axial algebras.
  
  Section~\ref{sec pre} recalls definitions---%
  in particular, the Seress property of the fusion rules of an idempotent $e$
  is that $e$ is globally associative with its $1,0$-eigenspaces---%
  and preliminary results.
  Section~\ref{sec 3trgps} presents Hypothesis~\ref{hyp vreg} on $3$-transposition groups,
  asking that maximal $3$-transposition subgroups act transitively on the transpositions they do not contain,
  and Theorem~\ref{thm vreg} showing that this holds in large classes of examples.
  In Section~\ref{sec idempots},
  we introduce Definition~\ref{def linidempot},
  the linear idempotents:
  idempotents which are identities of parabolic subalgebras,
  that is, come from $3$-transposition subgroups,
  closed under differences $e-f$ when $f$ is in the $1$-eigenspace of $e$.
  We then prove that the following weakening of associativity holds:
  
  \begin{introthm}[\ref{thm fazit seress}]
      Linear idempotents in $M_\al(\cal G)$ are Seress (over a suitable field)
      if the $3$-transposition $G$ group of $\cal G$
      satisfies Hypothesis~\ref{hyp vreg}.
  \end{introthm}
  
  The Seress property is well-known to hold for all idempotents in Jordan algebras by the Peirce decomposition and multiplication of eigenspaces.
  Similarly, it holds for all idempotents in weight-$2$ subalgebras of vertex algebras
  via the fusion rules of the Virasoro algebra and application of \cite{miyamoto}, Lemma 5.1.
  However, \cf \cite{matsuo} Proposition 3.3.8 and \cite{jordan3trgps}, these are far from including all Matsuo or axial algebras;
  to our knowledge, this paper is the first which handles a general class of idempotents in these nonassociative algebras.
  
  A particular application of Theorem~\ref{thm fazit seress}
  is that we can find those tori, \ie maximal associative subalgebras,
  which arise from chains of parabolic subgroups via identity elements of chains of parabolic subalgebras.
  The search for tori was initiated in \cite{alonso},
  and goes back to classical work on the Monster,
  including \cite{mn} and framed vertex algebras,
  although we believe they still merit further investigation.
  
  Specialising to the case of the Matsuo algebra of the symmetric group $\Sym(n)$ of degree $n$,
  in Section~\ref{sec eigvals an}, we achieve
  \begin{introthm}[\ref{thm idempots an}]
    The eigenvalues of primitive linear idempotents in $M_\al(\cal A_n^\pm)$ are determined.
  \end{introthm}
  \noindent
  When $\al = \frac{1}{4}$, by Theorem~\ref{thm dlmn}, due to \cite{dlmn},
  the specialisation of Theorem~\ref{thm idempots an} in Lemma~\ref{lem hrs an 1/4}
  determines the highest weights of the Virasoro algebra
  occuring at weight $2$ in the lattice vertex algebra of the root system $\sqrt2\rt A_n$,
  giving a proof of the results of \cite{lwhwyamada} that generalises to other situations.
  
  Finally in Section~\ref{sec invols} we consider $\rt D_n$,
  where we observe some new bijections between involutions in the group
  and idempotents in the algebra, via their fusion rules.
  \begin{introthm}[\ref{thm invols dn}]
    In $\Aut(W(\rt D_n)/\Z(W(\rt D_n)))$,
    $n-3$ conjugacy classes of involutions
    are realised as Miyamoto involutions of linear idempotents in $M_\al(\cal D_n)$.
  \end{introthm}
  Note that, as $\cal A_n^\pm = \cal D_{n+1}$,
  in these last two sections we are considering the same object
  from two (combinatorially) different points of view.
  
  I would like to thank Sergey Shpectorov for his supervision during my PhD,
  of which this work forms a part;
  I would also like to thank the referee,
  whose comments greatly improved the clarity of this paper.

\section{Preliminaries}
	\label{sec pre}
	
  \begin{definition}
    \label{def 3trgp}
    A {\em $3$-transposition group} $(G,D)$ consists of a group $G$
    generated by a set of involutions $D\subseteq G$ such that
    \begin{center}
      i. $D$ is closed under $G$-conjugation, \quad and \quad
      ii. for any $c,d\in D$, $\size{cd}\leq3$.
    \end{center}
  \end{definition}
  
  For further material on $3$-transposition groups,
  including their classification,
  we refer to \cite{asch3} or \cite{hall-genth3}.
  
  Recall that any $3$-transposition group $(G,D)$
  is uniquely characterised, up to center, by a
  
  \begin{definition}
    \label{def fs}
    A {\em Fischer space $\cal G$} is a graph
    whose lines are sets of vertices of size $3$
    such that, if $\ell_1,\ell_2$ are two distinct intersecting lines,
    \begin{enumerate}
      \itemsep0pt
      \item $\size{\ell_1\cap\ell_2}=1$ and
      \item their points span a subspace isomorphic to
        $\cal P_2^\vee$ or $\cal P_3$ in Figure~\ref{fig projpl}.
    \end{enumerate}
  \end{definition}
  
	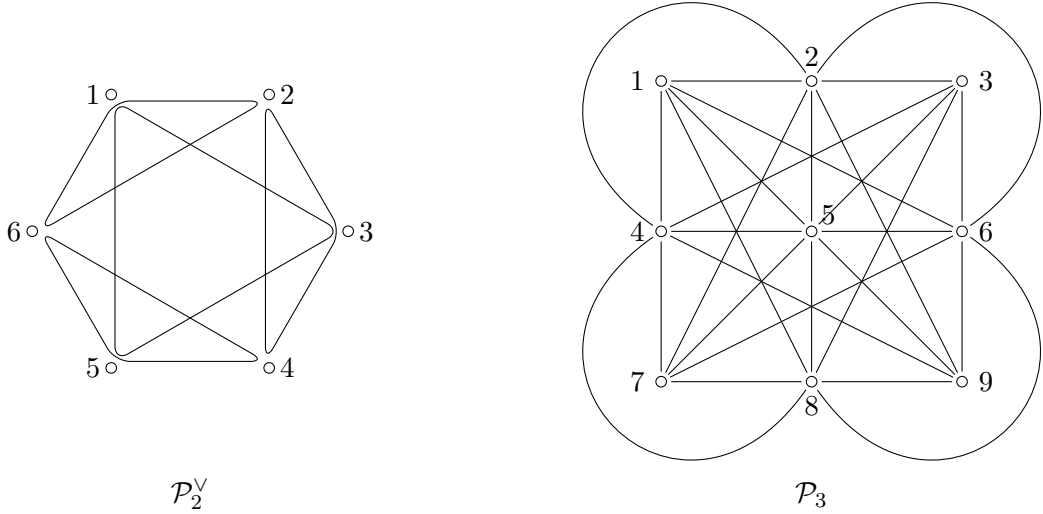
\begin{figure}[ht]
	\begin{center}
	\begin{tikzpicture}
		\useasboundingbox (-3.5cm,3.5cm) rectangle (3.5cm,-3.5cm);
		\node at (0,-3.5cm) {$\cal P_2^\vee$};

		\tikzstyle{every node}=[draw,circle,fill=white,minimum size=4pt,inner sep=0pt]
		\draw
			node at +(120:2.1cm) [label=left:$1$] {}
			node at +(60:2.1cm) [label=right:$2$] {}
			node at +(0:2.1cm) [label=right:$3$] {}
			node at +(300:2.1cm) [label=right:$4$] {}
			node at +(240:2.1cm) [label=left:$5$] {}
			node at +(180:2.1cm) [label=left:$6$] {}
		;

		\draw[rounded corners=0.2cm] (120:2cm) -- (60:2cm) -- (180:2cm) -- cycle;
		\draw[rounded corners=0.2cm] (60:2cm) -- (0:2cm) -- (300:2cm) -- cycle;
		\draw[rounded corners=0.2cm] (180:2cm) -- (300:2cm) -- (240:2cm) -- cycle;
		\draw[rounded corners=0.2cm] (120:2cm) -- (0:2cm) -- (240:2cm) -- cycle;

	\end{tikzpicture}
	\quad\quad\quad
	\begin{tikzpicture}
		\useasboundingbox	(-1.5cm,1.5cm) rectangle (5.5cm,-5.5cm);
		\node at (2cm,-5.5cm) {$\cal P_3$};

		\tikzstyle{every node}=[draw=white,ultra thick,
			circle,fill=white,minimum size=6pt,inner sep=0pt]
		\draw (0,0) node (1) [label=left:$1$] {}
			++(0:2cm) node (2) [label=above:$2$] {}
			++(0:2cm) node (3) [label=right:$3$] {}
			++(270:2cm) node (6) [label=right:$6$] {}
			++(180:2cm) node (5) [label=above right:$5$] {}
			++(180:2cm) node (4) [label=left:$4$] {}
			++(270:2cm) node (7) [label=left:$7$] {}
			++(0:2cm) node (8) [label=below:$8$] {}
			++(0:2cm) node (9) [label=right:$9$] {}
		;
		\tikzstyle{every node}=[draw,
			circle,fill=white,minimum size=4pt,inner sep=0pt]
		\draw (1) node {}
			(2) node {}
			(3) node {}
			(4) node {}
			(5) node {}
			(6) node {}
			(7) node {}
			(8) node {}
			(9) node {}
		;

		\draw (1) -- (2) -- (3);
		\draw (4) -- (5) -- (6);
		\draw (7) -- (8) -- (9);
		\draw (1) -- (4) -- (7);
		\draw (2) -- (5) -- (8);
		\draw (3) -- (6) -- (9);
		\draw (1) -- (5) -- (9);
		\draw (3) -- (5) -- (7);

		\draw (1) to (6);
		\draw[shift=(1)] (6) .. controls (330:7.5cm) and (300:7.5cm)	.. (8);
		\draw (8) to (1);

		\draw (3) to (4);
		\draw[shift=(3)] (4) .. controls (210:7.5cm) and (240:7.5cm)	.. (8);
		\draw (8) to (3);

		\draw (7) to (2);
		\draw[shift=(7)] (2) .. controls (60:7.5cm) and (30:7.5cm)	.. (6);
		\draw (6) to (7);

		\draw (9) to (2);
		\draw[shift=(9)] (2) .. controls (120:7.5cm) and (150:7.5cm)	.. (4);
		\draw (4) to (9);
	\end{tikzpicture}
	\end{center}
	\caption{
		The dual affine plane $\cal P_2^\vee$
		and the affine plane $\cal P_3$
	}
	\label{fig projpl}
	\end{figure}
	
	Namely, for $(G,D)$ a $3$-transposition group
	its Fischer space $\cal G$ has point set $D$
	and lines $\{c,d,e\}$ for any $c,d,e\in D$
	such that $\la c,d,e\ra\cong\Sym(3)$.
	
	Some interesting Fischer spaces are Weyl groups $W(\rt X_n)$ of root systems $\rt X_n$
	(for the latter two topics, we refer to \cite{carter});
  Recall that the Weyl group of $\rt A_n$ is $\Sym(n+1)$
  and the Weyl group of $\rt D_n$ is $2^n:\Sym(n+1)$.
  The set of reflections $D$ coming from the roots
  make $(W(\rt X_n),D)$ a $3$-transposition group in each case.
  When $\rt X_n$ is simply-laced, as $\rt A_n,\rt D_n$ are, we use $\cal X_n$
	to denote the Fischer space of its Weyl group.

  For distinct points $x,y$ in a Fischer space $\cal G$,
  we write $x\sim y$ if there exists a line in $\cal G$ containing both $x$ and $y$,
  and $x\not\sim y$ otherwise.
  If $x\sim y$, the line $\ell$ containing $x$ and $y$ is unique
  by i. of Definition~\ref{def fs}.
  We write $x\wedge y$ for $\ell = \{ x,y,x\wedge y \}$.
  
  In \cite{matsuo}, Matsuo introduced an algebra on Fischer spaces:
  
  \begin{definition}
    \label{def matsuo}
    The {\em Matsuo algebra $M_\al(\cal G)_R$}
    of the Fischer space $\cal G$ over the ring $R$ containing $\frac{\al}{2}$
    is the free $R$-module spanned by the points of $\cal G$
    together with the bilinear multiplication, for $x,y$ points of $\cal G$,
    \begin{equation}
      xy = \begin{cases}
        x & \text{ if } x = y, \\
        0 & \text{ if } x\not\sim y, \\
        \frac{\al}{2}(x + y - x\wedge y) & \text{ if } x\sim y.
      \end{cases}
    \end{equation}
  \end{definition}
  
  We view $\cal G$ as embedded in $M_\al(\cal G)_R$,
  so that $x\in\cal G$ is an idempotent, that is, $xx = x$.
  The following definitions come with a view towards the idempotents in Matsuo algebras.
  
  Suppose that $A$ is an algebra over a ring $R$.
  For arbitrary $x\in A$,
  write $\ad(x)$ for the adjoint map in $\End(A)$
  that is left-multiplication: $\ad(x)\colon y\to xy$.
  The eigenvalues, eigenvectors and eigenspaces of $x$
  are the eigenvalues, eigenvectors and eigenspaces of $\ad(x)$.
  The element $x$ is also said to be {\em diagonalisable}
  if $\ad(x)$ is diagonalisable as a matrix,
  that is, there exists a basis of $A$ consisting of eigenvectors of $x$.
  We write, for $\phi\in R$,
  \begin{equation}
    A^x_\phi = \{ y\in A \mid xy = \phi y \}
  \end{equation}
  for the $\phi$-eigenspace of $x$,
  and extend the notation so that, for $\Phi\subseteq R$ a set,
  \begin{equation}
    A^x_\Phi = \bigoplus_{\phi\in\Phi} A^x_\phi,
  \end{equation}
  including $A^x_\emptyset = 0$.
  
  \begin{definition}
    \label{def fusruls}
    {\em Fusion rules} are a set $\Phi\subseteq R$
    together with a symmetric map $\star\colon\Phi\times\Phi\to2^\Phi$.
    
    A diagonalisable idempotent $x\in A$ is a {\em $\Phi$-axis}
    if all of its eigenvalues lie in $\Phi$ and
    \begin{equation}
      A^x_\phi A^x_\psi \subseteq A^x_{\phi\star\psi} = \bigoplus_{\chi\in\phi\star\psi} A^x_\chi,
    \end{equation}
    that is, the product $yz$ of a $\phi$-eigenvector with a $\psi$-eigenvector
    is in the span of $\chi$-eigenvectors with $\chi\in\phi\star\psi$.
  \end{definition}
  
  \begin{definition}
    \label{def seress}
    Fusion rules $\Phi$ containing $0,1$ are {\em Seress}
    if $1\star\phi \subseteq \{\phi\} \supseteq 0\star\phi$ for all $\phi\in\Phi$.
    In particular, this means $1\star0=\emptyset$.
  \end{definition}
  
  We observe that an idempotent in an associative algebra
  has eigenvalues $1,0$ and its eigenvectors multiply according to
  $1\star1 = \{1\},0\star0=\{0\}$ and $1\star0=\emptyset$,
  so it satisfies Seress fusion rules.

  \begin{lemma}
    \label{lem weak ser}
    An element $e\in A$ has fusion rules $\Phi$
    which are Seress
    if and only if $e$ associates with its $1,0$-eigenspace $A^e_{\{1,0\}}$.
  \end{lemma}
  \begin{proof}
  		Suppose that $e\in A$ a $\Phi$-axis.
		Let $x,z\in A$ be arbitrary.
		By linearity, we may take $x\in A^e_\phi$.
		Then $ex = \phi x$ and in particular $(ex)z = \phi xz$.

		Observe that $xz\in A^e_\phi$ for any $x\in A^e_\phi,z\in A^e_{\{1,0\}}$,
		for any $\phi\in\Phi$,
		if and only if $\Phi$ is Seress.
		Furthermore $xz\in A^e_\phi$ if and only if $e(xz) = \phi xz$,
		that is, $e(xz) = (ex)z$.
  \end{proof}
  
  	\begin{table}[ht]
	\begin{center}
	\renewcommand{\arraystretch}{2}
	\setlength{\tabcolsep}{0.75em}
	\begin{tabular}{c|ccc}
			$\star$	& $1$ & $0$ & $\al$ \\
		\hline
			$1$ & $\{1\}$ & $\emptyset$ & $\{\al\}$ \\
			$0$ &	& $\{0\}$ & $\{\al\}$ \\
			$\al$ &	&	& $\{1,0\}$ \\
	\end{tabular}
	\end{center}
	\caption{Jordan fusion rules $\Phi(\al)$}
	\label{tbl jordan}
    \end{table}
	
  The Jordan fusion rules of Table \ref{tbl jordan} take a primary role in this work.
  It is not difficult to see that they are Seress, and 
  
  \begin{lemma}[\cite{hrs}, Theorem 6.2]
    \label{lem matsuo axes}
    Any point $x\in\cal G\subseteq M_\al(\cal G)_R$ is a $\Phi(\al)$-axis.
    \qed
  \end{lemma}

  We finally note
  	\begin{theorem}[\cite{dlmn} Theorem 3.1]
		\label{thm dlmn}
		The Matsuo algebra $M^{1/2}_{1/4}(\cal X_n^\pm)$
		(\cf Definition~\ref{def dbl}),
		modulo its radical (\cf \eqref{eq form}),
		is the weight-$2$ subalgebra
		of a vertex algebra $V_{\sqrt2\rt X_n}^+$
		when $\cal X_n$ is the Fischer space of a simply-laced root system $\rt X_n$,
		that is, $\rt A_n,n\in\NN$, $\rt D_n,n\geq4$, or $\rt E_6,\rt E_7,\rt E_8$.
		The radical of $M^{1/2}_{1/4}(\cal A_n^\pm)$ is $0$ for all $n$.
        \qed
	\end{theorem}
  
  	\begin{theorem}[Perron-Frobenius, \cite{agt} Theorem 8.8.1]
		\label{thm perfrob}
		For an irreducible matrix $A$ over $\CC$,
		there exists a real positive eigenvalue $\rho$ of $A$
		such that $\size\lm\leq\size\rho$ for all eigenvalues $\lm$ of $A$,
		and the $\rho$-eigenspace of $A$ is $1$-dimensional.
		\qed
	\end{theorem}
    Recall that the adjacency matrix of a connected graph is irreducible,
    and if $A$ is $k$-regular (that is, the neighbourhood of every point has size $k$)
    then $\rho = k$ in Theorem~\ref{thm perfrob}.

\section{$3$-transposition groups}
  \label{sec 3trgps}
  
  Suppose that $(G,D)$ is a $3$-transposition group.
  A subgroup $H\subseteq G$ is {\em parabolic} (also called a $D$-subgroup)
  if $H$ is generated by $H\cap D$.
  A {\em maximal parabolic subgroup} is a parabolic subgroup $H$
  maximal by inclusion among parabolic subgroups.
  We call $(G,D)$ {\em connected} if $D$ is a single conjugacy class,
  that is, $D = d^G$ for any $d\in D$.
  When $(G,D)$ is connected, so is the Fischer space $\cal G$ of $(G,D)$,
  and there exists a constant $k_\cal G\in\NN$ such that
  for any $x\in\cal G$,
  $\size{x^\sim} = \size{\{ y\in\cal G\mid x\sim y \}} = k_\cal G$,
  that is, as a graph $\cal G$ is $k_\cal G$-regular.
  
  The {\em boundary graph $\cal G/\cal H$}
  of an embedding $\cal H\subseteq\cal G$ of Fischer spaces
  is the graph with point set
  \begin{equation}
    \label{eq hsim}
    \cal H^\sim = \{ x\in\cal G\mid x\not\in\cal H, x\sim y\text{ for some }y\in\cal H \}
  \end{equation}
  and lines $\{x,y\}$ for all $x,y\in\cal G\smallsetminus\cal H$
  such that $x\wedge y\in\cal H$.
  
  \begin{definition}
    \label{def vreg}
    An embedding $\cal H\subseteq\cal G$ of Fischer spaces
    is {\em very regular} if $\cal H$ is a maximal subspace in $\cal G$,
    and $\cal H,\cal G,\cal G/\cal H$ are connected.
    
    A parabolic subgroup $H$ of a $3$-transposition group $(G,D)$
    is {\em very regular} if it induces a very regular embedding
    of Fischer spaces $\cal H\subseteq\cal G$.
  \end{definition}
  
  For a very regular embedding $\cal H\subseteq\cal G$
  it follows that $\cal G/\cal H$ is $k^\cal H_\cal G$-regular
  for some $k^\cal H_\cal G\in\NN$.

  We conjecture that an arbitrary maximal connected parabolic subgroup $H$
  of a connected $3$-transposition group $(G,D)$ is very regular;
  this holds for many known examples, as shown in the following Theorem~\ref{thm vreg}.
  
  Recall that $\affrt A_n$ is the affine extension of $\rt A_{n-1}$.
  For a Weyl group $W(\rt X_n)$,
  the transpositions are the conjugacy class of reflections of roots in $\rt X_n$.
  To define the group $G = W_k(\affrt A_n)$ for $k=2,3$,
  let $V$ be the vector space $\FF_k^{n+1}$ with basis $\{v_0,\dotsc,v_n\}$
  and $\hat G$ the semidirect product of $V$ with $\Sym(n+1)$ using the permutation action on the given basis.
  Then $G$ is the quotient $\hat G/\la v_0+\dotsm+v_n\ra$, and
  the transpositions in $G$ are the image of the conjugacy class of $(1,2)^{\Sym(n+1)}$.
  By $3^n:2$ we mean the elementary abelian group $3^n$ extended by an inverting involution,
  unless otherwise indicated.
  In all groups of shape $3^m:2$,
  the transpositions are the unique class of involutions.
  
  \begin{theorem}
    \label{thm vreg}
		The connected maximal parabolic subgroups $H$ of $(G,D)$
		induce very regular Fischer spaces $\cal H\subseteq\cal G$
		when $(G,D)$ is, for any $n\in\NN$,
		the Weyl group of $\rt A_n$, $\rt D_n$, $\rt E_6,\rt E_7,\rt E_8$,
		or $W_k(\affrt A_n)$ for $k=2,3$,
		or $3^n:2$,
		or M.~Hall's $3^{10}:2$.
	\end{theorem}
	\begin{proof}
	    Suppose that $\cal H,\cal G$ are connected.
	    We now show that the group-theoretic condition that
	    \begin{equation}
	        \label{eq gpthcond}
	        D=(H\cap D)\cup d^H
	        \text{ for any }
	        d\in D\smallsetminus H
	    \end{equation}
        implies that $\cal G/\cal H$ is connected:
	    namely, the conjugation action of $H$ on $d$
	    is afforded by its generators $H\cap D$,
	    and two elements $d,d'\in D$ are $H$-conjugate
	    if and only if the corresponding points $d,d'\in\cal G$
	    can be path-connected in $\cal G$ by lines nontrivially intersecting $\cal H$.
	    Thus, the proof of this theorem is reduced to showing that \eqref{eq gpthcond} holds in each case.

		Recall that the Weyl group $(G,D)$ of $\rt A_n$
		is $G = \Sym(n+1), D = (1,2)^G$.
		Let $E\subseteq D$ and $S\subseteq\{1,\dotsc,n+1\}$
		be the support of $E$,
		that is, the smallest subset $S$ of $\{1,\dotsc,n+1\}$
		such that any transposition $e\in E$
		is of the form $(s_1,s_2)$ for some $s_1,s_2\in S$.
		Then partition $S$ into orbits $S_1,\dotsc, S_n$ of $\la E\ra$.
		Observe that $\la E\ra \cong \Sym(\size{S_1})\times\dotsm\times\Sym(\size{S_n})$
		and therefore $E$ does not satisfy the hypothesis of connectedness
		unless $S = S_1$ is a single orbit.
		Furthermore if $\size S$ is less than $n$ then $H$ is not maximal.
		Therefore a connected maximal parabolic subgroup $H$ of $G$
		has support $\{1,\dotsc,j-1,j+1,\dotsc,n+1\}$ for some $j$
		and $H\cong\Sym(n)$.
		In these cases let $d = (1,j)$, or $d = (1,2)$ if $j=1$,
		so that $d\in D\smallsetminus(D\cap H)$.
		We see that $D = (H\cap D)\cup d^H$.
		
		As $W(\rt D_n)\cong W_2(\affrt A_{n-1})$ by \cite{hall-genth3},
		we cover it below as part of $W_k(\affrt A_{n-1})$.

		The cases for $W(\rt E_n), n = 6,7,8$,
		were checked in \cite{magma} with the computational assistance
		of Raul Moragues Moncho.

		Suppose that $(G,D)$ comes from $W_k(\affrt A_n)$ when $k=2,3$ and $n\geq3$.
		There are two possibilities for a parabolic subgroup $H$
		such that $H\cap D$ is a single conjugacy class:
		either $H$ is isomorphic to $\Sym(n)$ or to $W_k(\affrt A_{n-1})$.
		We use a representation of $G$ as a matrix group.
		Let
		\begin{equation}
			\begin{gathered}
				g_1 = \begin{pmatrix}
					0 & 1 \\
					1 & 0 \\
				\end{pmatrix}
				\oplus I_{n-1},\quad
				g_2 = \begin{pmatrix}
					1
				\end{pmatrix}\oplus
				\begin{pmatrix}
					0 & 1 \\
					1 & 0 \\
				\end{pmatrix}
				\oplus I_{n-2},\quad
				\dotsc,\quad
				g_{n-1} = I_{n-2}\oplus
				\begin{pmatrix}
					0 & 1 \\
					1 & 0 \\
				\end{pmatrix}
				\oplus\begin{pmatrix}
					1
				\end{pmatrix},\\
				g_{n+1} = \left(
				\begin{array}{cc|c|c}
					0 & 1 & 0 & 0 \\
					1 & 0 & 0 & 0 \\
					\hline
					0 & 0 & I_{r-3} & 0 \\
					\hline
					1 & -1 & 0 & 1 \\
				\end{array}
				\right),\quad
				h = \left(
				\begin{array}{c|c}
					I_{n+1} & 0 \\
					\hline
					1 & 1 \\
				\end{array}
				\right)
				\text{ over }
				\FF_k.
			\end{gathered}
		\end{equation}
		Then $G\cong \la g_1,\dotsc,g_{n+1}\ra/\la h\ra$
		and $D$ is the set of conjugates of $\{g_i\la h\ra\}_{1\leq i\leq n+1}$.
		We also set $\hat G = \la g_1,\dotsc,g_{n+1}\ra$
		and $\hat D$ the set of conjugates of $\{g_i\}_{1\leq i\leq n+1}$.
		Now $H\cong W_k(\affrt A_n)$
		if and only if, up to conjugation,
		$H = \hat H/\la h\ra$ for $\hat H = \la g_1,\dotsc,g_{n-1},g_{n+1}\ra$.
		Then it is clear that in $\hat G$,
		$\hat D = (\hat H\cap \hat D)\cup g_n^{\hat H}$.
		The same property descends to the quotient,
		so that $D = (H\cap D)\cup (g_n\la n\ra)^H$.
		This shows that $\cal G/\cal H$ is connected, so $\cal H\subseteq\cal G$ is very regular.
		The other possibility is that $H = \la g_1,\dotsc, g_n\ra/\la n\ra$.
		In this case, when $k = 2$ we see that $W_2(\affrt A_{n-1})\cong W(\rt D_n)$.
		We can observe that in general in $\hat G$,
		the orbit of $g_{n+1}$ under the action of $\hat H = \la g_1,\dotsc,g_n\ra$
		has size $\frac{1}{2}n(n+1)$ if $k = 2$ and $n(n+1)$ if $k = 3$,
		so that $\hat H$ is transitive on the transpositions in $\hat G$ outside $\hat H$.
		This again holds in the quotient $H$.

		When $G = 3^n:2$,
		there is only one conjugacy class $D$ of involutions.
		Observe that any subset of involutions of $G$
		generates a subgroup $H\cong 3^m:2$ for some $m$.
		Then $H$ is maximal if $m = n-1$.
		In this case, if $t,s$ are two transpositions in $D\smallsetminus H$,
		then $\la t,s\ra\cap H = \{t^s\}$ as $t^s\not\in H$ would contradict maximality,
		so $\cal G/\cal H$ is connected.
		This shows that it is also regular by transitivity.

		That the statement holds for M.~Hall's $G\cong 3^{10}:2$
		was checked in \cite{gap}
		using the presentation
        \begin{equation}
            \begin{aligned}
			G = \la a, b, c, d \mid & a^2 = b^2 = c^2 = d^2 = (ab)^3 = (ac)^3 = (ad)^3 = (bc)^3 = (bd)^3 = (cd)^3 \\
			& \quad = (b^cd)^3 = (a^bc)^3 = (a^bd)^3 = (a^cd)^3
				 = (a^{bd}c)^3 = (a^{cd}b)^3 = (a^{dc}b)^3 = 1\ra,
			\end{aligned}
        \end{equation}
		given in \cite{hall-genth3}, Proposition 2.9.
	\end{proof}
	
    By inductive application of Theorem~\ref{thm vreg}, we have that the hypothesis
    \begin{hypothesis}
        \label{hyp vreg}
        For any parabolic subgroups $K,H$ of $(G,D)$
        such that $K$ is a maximal parabolic subgroup of $H$,
        $K$ is very regular in $H$.
    \end{hypothesis}
    holds for many examples of $(G,D)$.

\section{Idempotents}
  \label{sec idempots}
  
  Let $\FF$ be a field containing $\frac{1}{2}$ and $\al$.
  In this section, we investigate an important class of idempotents,
  coming from identity elements of parabolic subalgebras,
  and show that they are well-behaved.
  
  \begin{lemma}
    \label{lem units}
    If $\cal G$ is a connected Fischer space
    and $\al\neq-\frac{2}{k_\cal G}$
    then $A = M_\al(\cal G)_\FF$ is unital, with
	\begin{equation}
	    \label{eq id}
		\id_\cal G = \frac{1}{1 + \frac{1}{2}\al k_\cal G}\sum_{x\in\cal G}x.
		\qedhere
	\end{equation}
  \end{lemma}
	\begin{proof}
		We show that, for $x\in\cal G$,
		\begin{equation}
			x\sum_{y\in\cal G}y = (1 + \frac{1}{2}\al k_\cal G)x.
		\end{equation}
		Observe that $\cal G = \{x\}\cup x^\sim\cup x^{\not\sim}$,
		for $x^\sim = \{y\in\cal H\mid x\sim y\}$
		and $x^{\not\sim} = \{y\in\cal H\mid x\not\sim y\}$;
		also, $\size{x^\sim} = k_\cal G$.
		Then
		\begin{equation}
			x\sum_{y\in\cal G}y
			= xx + x\sum_{y\in x^\sim}y + x\sum_{y\in x^{\not\sim}}y
			= x + \frac{\al}{2}\sum_{y\in x^\sim}(x+y-x\wedge y) + 0
			= (1+\frac{1}{2}\al k_\cal G)x,
		\end{equation}
		where the last equality follows since,
		as $y$ ranges over $x^\sim$, so does $x\wedge y$:
		that is, $\{x\wedge y\mid y\in x^\sim\} = x^\sim$
		and $\sum_{y\in x^\sim}(y - x\wedge y) = \sum_{y\in x^\sim}y - \sum_{y\in x^\sim}x\wedge y = 0$.

		If $1+\frac{1}{2}\al k_\cal G$ is nonzero, that is, $\al\neq-\frac{2}{k_\cal G}$,
		it follows that $\id_\cal G$ of \eqref{eq id} is the identity of $A$.
	\end{proof}

	This result, and its sequels, admits generalisation to nonconnected Fischer spaces $\cal G$.
	If $\cal G = \cal G_1\cup\dotsm\cup\cal G_n$ is a partition
	into pairwise disconnected Fischer spaces such that each $M_\al(\cal G_i)$ is unital,
	then $\id_{\cal G} = \sum_i\id_{\cal G_i}$.
	
	Write $\Spec(M)$ for the eigenvalues of a matrix $M$,
	and by extension $\Spec(\cal G)$ for the eigenvalues
	of the adjancency matrix of a graph $\cal G$.
	Also recall the Minkowski addition, and difference, of sets: $X \pm Y = \{x\pm y\mid x\in X,y\in Y\}$.
	For $x\in\FF$ and $Y\subseteq\FF$,
	we will write $x-Y = \{x\}-Y$.
	Now we can describe the eigenvalues of $\id_\cal H$.
	
	\begin{lemma}
		\label{lem id eigval}
		Suppose that $\cal H\subseteq\cal G$ is very regular and $\al\neq-\frac{2}{k_\cal H}$.
		Then $\id_\cal H$ in $A = M_\al(\cal G)_\FF$ acts diagonalisably on the subspaces spanned by
		\begin{align*}
			& \cal H \text{ with eigenvalue } 1, \text{ and } \\
			& \cal H \cup \cal H^\sim \text{ with further eigenvalues } \frac{\al}{2+\al k_\cal H}(k_\cal H^\cal G - \Spec(\cal G/\cal H)).
		\end{align*}
	\end{lemma}
	\begin{proof}
	    By Lemma~\ref{lem units}, $\la\cal H\ra_\FF$ is a subspace of the $1$-eigenspace of $\id_\cal H$.
	    
		Take $y\in\cal H^\sim$, where $\cal H^\sim$ is defined in \eqref{eq hsim}. Then $y\not\in\cal H$ and
		\begin{align}
			\label{eq-action-of-id}
			\id_\cal H y & = \frac{1}{1+\frac{1}{2}\al k_\cal H}\sum_{x\in y^\sim\cap\cal H}\frac{\al}{2}(x + y - x\wedge y).
		\intertext{
		If $x\in\cal H$ and $y\in x^\sim\smallsetminus\cal H$
		then $y,x\wedge y\in\cal H^\sim$, so that $\id_\cal H$
		fixes the subspace spanned by $\cal H\cup\cal H^\sim$.
        Furthermore, as $k_\cal G^\cal H = \size{y^\sim\cap\cal H}$,
		}
		    \label{eq idhy}
			\id_\cal H y & =
			\frac{\al k_\cal G^\cal H}{2+\al k_\cal H}y
			+\frac{\al}{2+\al k_\cal H}\sum_{x\in y^\sim\cap\cal H}(x-x\wedge y).
		\end{align}
		Observe that $x\in\cal H$ and $x\wedge y\in\cal H^\sim$ (for, if $x\wedge y\in\cal H$, then as $\cal H$ is a subspace we would have $x\wedge(x\wedge y)=y\in\cal H$).
		Now suppose that $e\in\la\cal H\cup\cal H^\sim\ra_\FF$ is an eigenvector for $\id_\cal H$.
		Write $e_\sim$ for the projection of $e$ to $\la\cal H^\sim\ra_\FF$ and $e_0 = e - e_\sim$.
		Then
		\begin{equation}
			\id_\cal H e = e_0
				+ \id_\cal H e_\sim = \lm e
		\end{equation}
		for some $\lm$,
		and, using \eqref{eq idhy},
		the projection of $\id_\cal H e_\sim$ to $\la\cal H^\sim\ra_\FF$ is
		\begin{equation}
			\frac{\al}{2+\al k_\cal H}(k_\cal G^\cal HI_{\size{\cal H}}
				+ \ad(\cal G/\cal H))e_\sim = \lm e_\sim,
		\end{equation}
		where $\ad(\cal G/\cal H)y = \sum_{x\in y^\sim\cap\cal H}x\wedge y$ is extended $\FF$-linearly to $\la\cal H^\sim\ra_\FF$.
        Therefore if $e_\sim\neq0$,
        then $\lm$ is an eigenvalue of
		\begin{equation}
			\frac{\al}{2+\al k_\cal H}(k_\cal G^\cal H I_{\size{\cal H}}-\ad(\cal G/\cal H)).
		\end{equation}
		Therefore $\lm$ is in $\frac{\al}{2+\al k_\cal H}(k_\cal G^\cal H-\Spec(\cal G/\cal H))$.
		By comparing dimensions, the eigenspaces of $\id_\cal H$ span $\la\cal H\cup\cal H^\sim\ra_\FF$,
		so $\id_\cal H$ is diagonalisable.
	\end{proof}

    This enables us to show that (the adjoint operator of) $\id_\cal H$ is diagonalisable:

	\begin{lemma}
		\label{lem id diag}
		Suppose that $\cal H\subseteq\cal G$
		and $\cal H$ is very regular in any parabolic subspace $\cal G'\subseteq\cal G$
		in which $\cal H$ is maximal.
		If $\al\neq-\frac{2}{k_\cal H}$,
		then $\id_\cal H$ is diagonalisable in $M_\al(\cal G)$.
	\end{lemma}
	\begin{proof}
		By Lemma~\ref{lem units},
		the subalgebra of $M_\al(\cal G)$ spanned by $\cal H$
		has an identity $\id_\cal H$.

		Let $x\in\cal G\smallsetminus\cal H$ be arbitrary
		and set $\cal G' = \la x,\cal H\ra$.
		If $x\not\in\cal H^\sim$,
		then $\id_\cal H x = 0$;
		otherwise, $\id_\cal H$ acts on $\cal G'$ diagonalisably
		by Lemma~\ref{lem id eigval}.
		Now $\cal G\smallsetminus\cal H$
		can be partitioned in $\cal G'_1,\cal G'_2,\dotsc,\cal G'_r$ and $\cal H^{\not\sim}$
		where each $\cal G'_i$ is a subgraph of $\cal G$
		in which $\cal H$ is maximal.
		That $\cal G'_i\cap\cal G'_j = \cal H$ if $i\neq j$
		follows from the fact that,
		if $y\in(\cal G'_i\cap\cal G'_j)\smallsetminus\cal H$
		then $\cal G'_i = \la\cal H,y\ra = \cal G'_j$ by maximality,
		so the $\cal G'_i$ have pairwise trivial intersection.
		Thus $\id_\cal H$ acts diagonalisably on a basis of $M_\al(\cal G)$.
	\end{proof}

    We can also classify the $1$- and $0$-eigenspaces of any $\id_\cal H$.

	\begin{lemma}
		\label{lem units 10eigvect}
		Suppose that $\cal H\subseteq\cal G$ is very regular.
		If $\al$ is an indeterminate over $\FF$
		and $A = M_\al(\cal G)_{\FF(\al)}$,
		then the $1$-eigenspace of $\id_\cal H$ is $\la\cal H\ra_{\FF(\al)}$,
		the $0$-eigenspace is $1$-dimensional,
		and these are the only eigenvalues of $\id_\cal H$ contained in $\FF\subseteq\FF(\al)$.
		If $\al\in\FF$, the $1,0$-eigenspaces of $\id_\cal H$ in $M_\al(\cal G)_\FF$
		are the same if $\al\neq\frac{2}{k_\cal G^\cal H-k_\cal H-\lm}$ for any $\lm\in\Spec(\ad(\cal G/\cal H))$.
	\end{lemma}
	\begin{proof}
		The eigenvalues of $\id_\cal H$ on $\cal G$ are classified
		by Lemma~\ref{lem id eigval},
		showing that if $\al$ is an indeterminate,
		then $1$ and $0$ are the only eigenvalues in $\FF\subseteq\FF(\al)$.
        The eigenvalues of $\id_\cal H$ are $1$ and $\frac{\al}{2+\al k_\cal H}(k_\cal G^\cal H-\lm)$,
		for $\lm\in\Spec(\cal G/\cal H)$.
		Evidently $\cal H\subseteq A^{\id_\cal H}_1$.
		By Theorem~\ref{thm perfrob},
		the $k^\cal G_\cal H$-eigenspace of $\ad(\cal G/\cal H)$ is $1$-dimensional,
		so when $\al\neq0$ the $0$-eigenspace of $\id_\cal H$ is also $1$-dimensional.
		It only remains to consider other $1$-eigenvectors.
		The only solution to the equation, if $\lm\neq k^\cal H_\cal G$,
		\begin{equation}
			\frac{\al}{2+\al k_\cal H}(k_\cal G^\cal H-\lm) = 1
		\end{equation}
		is $\al = 2/(k_\cal G^\cal H - k_\cal H - \lm)$.
	\end{proof}

	We say that an element $x\in A$ is {\em Seress}
	if it acts diagonalisably
	and the fusion rules $\Phi$ satisfied by its eigenspaces are Seress as in Definition~\ref{def seress}.
	In particular, this applies to $\id_\cal H$ in certain cases:

	\begin{lemma}
		\label{lem units seress}
		If $\cal H\subseteq\cal G$ are very regular Fischer spaces
		and $\al\neq-\frac{2}{k_\cal H},\frac{2}{k_\cal G^\cal H-k_\cal H-\lm}$ for any $\lm\in\Spec(\ad(\cal G/\cal H))$,
		then $\id_\cal H$ is Seress in $M_\al(\cal G)_\FF$.
		Furthermore $\id_\cal H$ is Seress in $M_\al(\cal G)$
		if $\cal H\subseteq\cal G'$ is very regular
		whenever $\cal H\subseteq\cal G'$ is maximal
		and $\cal G'\subseteq\cal G$,
		and $\al\neq\frac{2}{k_\cal G^\cal H-k_\cal H-\lm}$ for $\lm\in\Spec(\ad(\cal G'/\cal H))$.
	\end{lemma}
	\begin{proof}
		Lemma~\ref{lem id eigval} showed that $\id_\cal H$ acts diagonalisably.
		We use the classification of $1$- and $0$-eigenvectors of Lemma~\ref{lem units 10eigvect}
		to prove that $1\star\phi\subseteq\{\phi\}\supseteq0\star\phi$
		for all eigenvalues $\phi$ of $\id_\cal H$ in $A = M_\al(\cal G)$
		(which in particular implies $1\star0 = \emptyset$).
		We first take the case when $\cal H\subseteq\cal G$ is very regular.

		Since under our hypotheses the $1$-eigenspace of $\id_\cal H$ is spanned by $\cal H$,
		which is closed under multiplication,
		we already have that $1\star1=\{1\}$.

		We will use four facts for the sequel.
		Firstly, observe that
    	\begin{equation}
    		\label{eq miyamatsuo}
    		y^{\tau(x)} = \begin{cases}
    			x\wedge y & \text{ if } x\sim y, \\
    			y & \text{ otherwise}.
    		\end{cases}
    	\end{equation}
		Secondly, for any $h\in\cal H,a\in A$,
		by application of \eqref{eq miyamatsuo},
		\begin{equation}
			\label{eq genericmult}
			ha = \frac{\al}{2}(\lm_h h + a - a^{\tau(h)})
			\text{ for some }\lm_h\in\FF.
		\end{equation}
		Thirdly, if $t\in\Aut(A)\subseteq\End(A)$ fixes $a\in A$,
		then $t$ centralises $\ad(a)\in\End(A)$ and the eigenspaces $A^a_\phi$ of $a$.
		Fourthly, if $h\in\cal H$ then, as $\cal H$ is closed under $\wedge$,
		$\tau(h)$ permutes the points of $\cal H$
		and therefore fixes $\id_\cal H$.
		
		To show that $1\star\phi=\{\phi\}$ for $\phi\neq1$,
		suppose that $h\in\cal H$ and $y$ is a $\phi$-eigenvector of $\id_\cal H$ in $A$.
		Set $y = y_0 + y_\sim$,
		for $y_\sim\in\la\cal H^\sim\ra_\FF$ the projection of $y$ onto the subspace spanned by $\cal H^\sim$
		and $y_0 = y - y_\sim$.
		Now as $y$ is a $\phi$-eigenvector for $\id_\cal H$,
		$\id_\cal H y = \phi y$ is again a $\phi$-eigenvector.
		On the other hand, using Lemma~\ref{lem units} and \eqref{eq genericmult},
		\begin{equation}
			\phi y = \id_\cal H y = \frac{1}{1+\frac{1}{2}\al k_\cal H}\sum_{h\in\cal H}hy
				= \frac{\al}{2+\al k_\cal H}\sum_{h\in\cal H}(\lm_h h + y - y^{\tau(h)}).
		\end{equation}
		Noting that $y^{\tau(h)}\in(A^{\id_\cal H}_\phi)^{\tau(h)} = A^{\id_\cal H}_\phi$,
		by rearranging terms we have an expression for $\sum_{h\in\cal H}\lm_h h$
		in terms of $\phi$-eigenvectors.
		On the other hand, any $h\in\cal H$ is a $1$-eigenvector
		and $A^{\id_\cal H}_1\cap A^{\id_\cal H}_\phi = 0$,
		so that $\sum_{h\in\cal H}\lm_h h = 0$.
		As the points $h\in\cal H$ are linearly independent,
		this means $\lm_h = 0$ for all $h\in\cal H$.
		Therefore $hy = \frac{\al}{2}(y - y^{\tau(h)})\in A^{\id_\cal H}_\phi$,
		so $1\star\phi=\{\phi\}$.

		To show that $1\star0=\emptyset$,
		observe that the $0$-eigenspace of $\id_\cal H$ is $1$-dimensional by Lemma~\ref{lem units 10eigvect},
		and fixed by any automorphism $t$ fixing $\id_\cal H$.
		In particular, $\tau(h)$ fixes $y\in A^{\id_\cal H}_0$,
		so by the previous paragraph, $hy = \frac{\al}{2}(y - y) = 0$.

		Therefore a $0$-eigenvector $z$ of $\id_\cal H$ in $M_\al(\cal G)$
		is also a $0$-eigenvector of any $h\in\cal H$.
		By Lemma~\ref{lem weak ser},
		for any $x\in A$ we have $h(xz) = (hx)z$.
		As $\id_\cal H$ is a linear combination of $h\in\cal H$,
		we conclude $\id_\cal H(xz) = (\id_\cal H x)z$.
		Thus $\id_\cal H$ and $z$ associate,
		and using the other direction of Lemma~\ref{lem weak ser}
		this implies that $0\star\phi=\{\phi\}$ for all $\phi\neq1$.

		We now tackle the general case of connected $\cal H$ in some $\cal G$
		such that $\cal H\subseteq\cal G'$ is very regular in every $\cal G'\subseteq\cal G$
		for which $\cal H\subseteq\cal G'$ is maximal.
		The $1$-eigenspace of $\id_\cal H$ in $M_\al(\cal G)$
		is still spanned by $\cal H$ and, by the same argument
		as that in the proof of Lemma~\ref{lem id diag},
		any $\phi$-eigenvector
		can be decomposed into a sum of $\phi$-eigenvectors
		lying in $M_\al(\cal G')$ for $\cal H\subseteq\cal G'$ very regular,
		unless $\phi=0$, in which case
		the $0$-eigenspace also includes $\cal H^{\not\sim}$.
		Therefore the fusion rules $1\star\phi=\{\phi\}$ for $\phi\neq0$ are satisfied.

		Suppose that $z\in\cal H^{\not\sim}$;
		then for $h\in\cal H$, $h\not\sim z$ so $hz = 0$.
		This shows that $1\star0=\emptyset$ in $M_\al(\cal G)$.

		To show that $0\star\phi=\{\phi\}$ in $M_\al(\cal G)$,
		we repeat our observation that the $0$-eigenvectors of $\id_\cal H$
		are $0$-eigenvectors of $h\in\cal H$, which are Seress,
		so that by linearity $\id_\cal H$ associates with its $0$-eigenspace
		and, using Lemma~\ref{lem weak ser}, therefore $0\star\phi=\{\phi\}$ for all $\phi\neq1$.
	\end{proof}

  \begin{lemma}
    \label{lem coset ax}
    Suppose that $e,f$ are idempotents
    and $f\in A^e_1$. Then
    \begin{enumerate}
      \itemsep0pt
      \item $e - f$ is an idempotent;
      \item if further $f$ is Seress and $e,f$ are diagonalisable,
        then $e - f$ is diagonalisable;
      \item if further $e$ is Seress and $A^{e-f}_{\{1,0\}}\subseteq A^e_{\{1,0\}}\cap A^f_{\{1,0\}}$,
        then $e - f$ is Seress.
    \end{enumerate}
	\end{lemma}
	\begin{proof}
    i. This is
      \begin{equation}
        (e-f)(e-f) = ee - 2ef + ff = e - 2f + f = e - f.
      \end{equation}
  
    ii. As $f$ is Seress, it associates with its $1$-eigenspace,
      in particular, with $e$.
      Therefore $(ex)f = e(xf)$ for all $x\in A$;
      equivalently, $\ad(e)\ad(f) = \ad(f)\ad(e)\in\End(A)$,
      so that $[\ad(e),\ad(f)] = 0$.
      Thus $\ad(e),\ad(f)$ are two commuting diagonalisable matrices,
      so they are simultaneously diagonalisable
      and their difference $\ad(e) - \ad(f) = \ad(e-f)$ is diagonalisable.
      
    iii. If $A^{e-f}_{\{1,0\}}\subseteq A^e_{\{1,0\}}\cap A^f_{\{1,0\}}$,
        then any $1$-eigenvector $x$ of $e-f$ is a $1$-eigenvector of $e$ and a $0$-eigenvector of $f$,
        and a $0$-eigenvector $z$ of $e-f$ is a $0$-eigenvector of $e$ and $f$.
        Since both $e$ and $f$ associate with their $1,0$-eigenspaces,
        $e-f$ associates with $A^e_{\{1,0\}}\cap A^f_{\{1,0\}}$.
        Therefore $e-f$ associates with $A^{e-f}_{\{1,0\}}$,
        and by Lemma~\ref{lem weak ser},
        $e-f$ is Seress.
	\end{proof}
	
  Observe that the hypotheses of iii. hold
  if, whenever $\lm,\mu$ are eigenvalues of $e,f$ respectively with $\lm-\mu\in\{1,0\}$,
  then $\mu = 0$.
  This is key to our last definition and theorem.
  
  \begin{definition}
    \label{def linidempot}
    Let $\cal G$ be a Fischer space and $A = M_\al(\cal G)_\FF$ its Matsuo algebra.
    Write $L_0$ for the set of identity elements of parabolic subalgebras.
    The set $L$ of {\em linear idempotents} of $A$
    is the minimal set containing $L_0$
    such that, for all $e,f\in L$ with $f\in A^e_1$, also $e-f\in L$.
  \end{definition}
	
  \begin{theorem}
    \label{thm fazit seress}
    Suppose that $(G,D)$ is a $3$-transposition group satisfying Hypothesis~\ref{hyp vreg}
    and set $A_\FF = M_\al(\cal G)_\FF$.
    The linear idempotents in $A_{\FF(\al)}$ are Seress when $\al$ is indeterminate over $\FF$.
  \end{theorem}
  \begin{proof}
    By Lemmas~\ref{lem id diag} and~\ref{lem units seress}, the identities of parabolic subalgebras are diagonalisable and Seress when $\al$ is an indeterminate
    (as this rules out any coincidences of eigenvalues such as $\al = -\frac{2}{k_\cal H},\frac{2}{k_\cal G^\cal H-k_\cal H-\lm}$).
    Suppose that $e,f\in L$, the linear idempotents of $A$, and that $f\in A^e_1$.
    Then it follows by ii. of Lemma~\ref{lem coset ax} that $e-f$ is diagonalisable with eigenvalues $\Spec(e)-\Spec(f)$.
    Therefore, to show that iii. of Lemma~\ref{lem coset ax} holds for an arbitrary linear idempotent $e$,
    which can be written as a sum $e=\sum_{i=1}^n(-1)^{i+1}\id_i$ for $\id_i$ the identity of a parabolic subalgebra,
    we need to consider when sums $\sum_{i=1}^n(-1)^{i+1}\lm_i$ of eigenvalues $\lm_i$ of $\id_i$ can equal $1$ or $0$.
    
    Contributions of eigenvalue $0$ coming from a constituent term $\id_i$ can be neglected.
    For a simultaneous eigenvector, as the $1$-eigenspaces satisfy the inclusions $A^i_1\subseteq A^{i+1}_1$,
    only the first $m$ consecutive idempotents may take eigenvalue $1$ for some $m\leq n$.
    As the sum is alternating, these contributions cancel to either $1$ or $0$.
    Therefore observe that an eigenvalue of $e$ is
    \begin{equation}
        \lm = \al\sum_{i=1}^n\frac{\mu_i}{2+\al k_i}
        \quad\text{ or }
        1-\lm,
    \end{equation}
    where $\mu_i,k_i\in\ZZ$;
    here, $\id_i$ is the identity of a subalgebra $A^{\id_i}_1$
    with Fischer space $\cal H_i\subseteq\cal G$ which is $k_i$-regular,
    and $\mu_i = (-1)^{i+1}(k_{\cal G'}^{\cal H_i}-\lm)$ for some very regular embedding $\cal H_i\subseteq\cal G'\subseteq\cal G$ and $\lm\in\Spec(\ad(\cal G'/\cal H_i))$.
    We solve for $\lm$ or $1-\lm$ equal to $1,0$;
    without loss of generality, we need only to find when $\lm = 1,0$.
    
    Comparing degrees of $\al$ in the expression for the numerator and denominator,
    we see that the denominator has a constant term,
    whereas the term of lowest degree in the numerator has degree $1$.
    Therefore they cannot be equal, so that the expression cannot evaluate to $\lm=1$.
    
    The other possibility is that $\lm = 0$, hence $\sum\frac{\mu_i}{2+\al k_i}=0$, which we now rule out.
    The denominators $2+\al k_i$ are all different,
    as the maximal, or Perron-Frobenius, eigenvalues $k_i,k_i+1$ of graphs $\cal H_i\subseteq\cal H_{i+1}$
    satisfy $k_i<k_{i+1}$ when $\cal H_i$ is strictly smaller than $\cal H_{i+1}$,
    which must be the case as $\id_i\neq\id_{i+1}$.
    Now the collection $\{\frac{1}{2+\al k_i}\}_{1\leq i\leq n}\subseteq\FF(\al)$ is linearly independent over $\FF$,
    as $\sum_{i=1}^n\frac{\mu_i}{2+\al k_i} = 0$ if and only if $\sum_{i=1}^n\mu_i\prod_{j\neq i}(2+\al k_j) = 0$,
    and by specialising $\al\mapsto-\frac{2}{k_i}$, since there are no repeated factors
    we see $\mu_i = 0$ in this sum.
    
    Thus indeed $A^{e-f}_{\{1,0\}}\subseteq A^e_{\{1,0\}}\cap A^f_{\{1,0\}}$, so $e-f$ is Seress
    by application of iii. of Lemma~\ref{lem coset ax}.
  \end{proof}

\section{Eigenvalues in $M_\al(\cal A_n)$}
  \label{sec eigvals an}
  
  For the results of this section we first present a graph construction.
  \begin{definition}
    \label{def dbl}
    Suppose that $\cal G$ is a Fischer space.
    Its {\em double graph $\cal G^\pm$}
    is the graph with point set $\{ x^+,x^- \mid x\in\cal G \}$
    and lines $\{ x^\ep, y^\eta, (x\wedge y)^{\ep\eta} \}$
    for any $x\sim y$ in $\cal G$, $\ep,\eta\in\{+,-\}$.
  \end{definition}

	\begin{lemma}
		\label{lem hat an is dn+1}
		The double graph of $\cal A_n$ is $\cal D_{n+1}$.
	\end{lemma}
	\begin{proof}
		Suppose that $\{x_1,\dotsc,x_m\}$ are the points in $\cal A_n$,
		inducing transpositions $\{t_1,\dotsc,t_m\}$ in the Weyl group $W(\rt A_n)$.
		Then there are $s_1,\dotsc,s_n$ transpositions among them
		satisfying the Coxeter presentation for $W(\rt A_n)$ in Figure~\ref{fig-an}.
		\begin{figure}[ht]
		\begin{center}
		\begin{tikzpicture}[scale=1.5]
			\tikzstyle{every node}=[draw=white,ultra thick,
				circle,fill=white,minimum size=8pt,inner sep=0pt]
			\draw[dashed,thick] (0:1.5cm) -- ++(0:1cm);
			\draw[thick] (0:1cm) -- ++(0:0.5cm);
			\draw[thick] (0:2.5cm) -- ++(0:0.5cm);
			\draw[thick]
				(0:0) node [label=below:$s_1$] {} --
				++(0:1cm) node [label=below:$s_2$] {}
				++(0:2cm) node [label=below:$s_n$] {}
			;
			\tikzstyle{every node}=[draw,
				circle,fill=white,minimum size=5pt,inner sep=0pt]
			\draw[thick]
				(0:0) node {}
				++(0:1cm) node {}
				++(0:2cm) node {}
			;
		\end{tikzpicture}
		\end{center}
		\caption{Coxeter presentation for $W(\rt A_n)$}
		\label{fig-an}
		\end{figure}
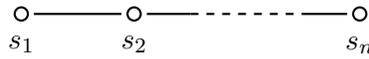

		Let $x_1^+,\dotsc,x_m^+,x_1^-,\dotsc,x_m^-$
		be the points of $\cal A_{n+1}^\pm$
		and $t_i^\ep$ the transposition $\tau(x_i^\ep)$ of $x_i^\ep$
		in the permutation representation.
		Then it follows that $S = \{t_1^-,t_1^+,t_2^+,t_3^+,\dotsc,t_r^+\}$,
		transpositions induced from the points of $\cal A^\pm_n$,
		satisfies the Coxeter presentation for $W(\rt D_{n+1})$ in Figure~\ref{fig-dn}.
		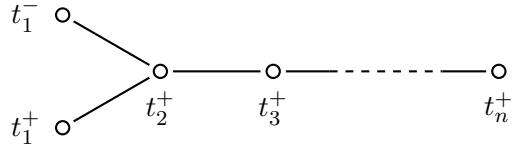
\begin{figure}[ht]
		\begin{center}
		\begin{tikzpicture}[scale=1.5]
			\tikzstyle{every node}=[draw=white,ultra thick,
				circle,fill=white,minimum size=8pt,inner sep=0pt]
			\draw[dashed,thick] (30:1cm) ++(0:1.5cm) -- ++(0:1cm);
			\draw[thick] (30:1cm) ++(0:1cm) -- ++(0:0.5cm);
			\draw[thick] (30:1cm) ++(0:2.5cm) -- ++(0:0.5cm);
			\draw[thick]
				(90:1cm) node [label=left:$t_1^-$] {} -- ++(330:1cm);
			\draw[thick]
				(0:0) node [label=left:$t_1^+$] {} --
				++(30:1cm) node [label=below:$t_2^+$] {} --
				++(0:1cm) node [label=below:$t_3^+$] {}
				++(0:2cm) node [label=below:$t_n^+$] {}
			;
			\tikzstyle{every node}=[draw,
				circle,fill=white,minimum size=5pt,inner sep=0pt]
			\draw[thick]
				(0:0) node {}
				++(90:1cm) node {}
				++(330:1cm) node {}
				++(0:1cm) node {}
				++(0:2cm) node {}
			;
		\end{tikzpicture}
		\end{center}
		\caption{Coxeter presentation for $W(\rt D_{n+1})$}
		\label{fig-dn}
		\end{figure}

		Moreover, $S$ generates $G = W(\rt A^\pm_n)$,
		so $G$ is a quotient of $W(\rt D_{n+1})$.
		In fact a counting argument shows that $G = W(\rt D_{n+1})$,
		since $W(\rt D_{n+1})$ has $n(n+1)$ transpositions
		and $W(\rt A_n)$ has $2\cdot\frac{1}{2}n(n+1)$, the same number.
		The corresponding points $x_1^-,x_1^+,x_2^+,\dotsc,x_r^+$
		generate $\cal A^\pm_n$,
		therefore $\cal D_{n+1} \cong \cal A^\pm_n$.
	\end{proof}

	\begin{lemma}
		\label{lem co an}
		The boundary graph
		$\cal A_n/\cal A_{n-1}$ is $K_n$,
		the complete graph on $n$ points.
	\end{lemma}
	\begin{proof}
		Recall that the group $W(\rt A_n)$
		generated by Miyamoto involutions of points $x\in\cal A_n$
		is the symmetric group $\Sym(n+1)$ on $n+1$ letters.
		Taking the embedding $H = \Sym(n)\subseteq\Sym(n+1) = G$
		that corresponds to $\cal A_{n-1}\subseteq\cal A_n$
		gives that $H$ has support $\{1,\dotsc,n\}$
		and $G$ has support $\{1,\dotsc,n+1\}$
		in the standard permutation realisation of $G$.
		Then if $s,t\in G\smallsetminus H$ are transpositions,
		they each move two letters in $\{1,\dotsc,n+1\}$.
		If $s$ moves two letters in $\{1,\dotsc,n\}$ then $s\in H$,
		so $s$ moves $n+1$; the same goes for $t$.
		We can therefore write $s = (i,n+1)$ and $t = (j,n+1)$
		for $1\leq i,j\leq n$.
		Then $st = (i,j)$ lies in $H$.
		This shows that the points $x,y\in\cal A_n$ corresponding to $s,t$
		satisfy $x\wedge y\in\cal A_{n-1}$.
		As $s,t$ were arbitrary,
		any two points in $\cal A_n/\cal A_{n-1}$ are connected.
	\end{proof}
	
	\begin{lemma}
		\label{lem dbl boundary}
		The double graph $(\cal G/\cal H)^\pm$ of $\cal G/\cal H$,
		for $\cal G,\cal H$ linear $3$-graphs,
		is ${\cal G^\pm}/\cal H^\pm$.
	\end{lemma}
	\begin{proof}
		The naive bijection works:
		take $x^\ep\in(\cal G/\cal H)^\pm$.
		Then $x\in\cal G/\cal H$
		and is uniquely identified with a point $x'$
		in $\cal G\smallsetminus\cal H$,
		for which there exists $y'\in\cal G\smallsetminus\cal H$
		with $x'\wedge y'\in\cal H$.
		Now $x'^\ep,y'^\ep\in{\cal G^\pm}\smallsetminus\cal H^\pm$
		and $x'^\ep\wedge y'^\ep\in\cal H^{\ep\ep}\subseteq\cal H^\pm$,
		so $x'^\ep\in{\cal G^\pm}/\cal H^\pm$.
		Therefore $(\cal G/\cal H)^\pm$
		has the same cardinality as ${\cal G^\pm}/\cal H^\pm$.
		Indeed identifying $y'\in\cal G\smallsetminus\cal H$
		in the above argument with $y\in\cal G/\cal H$
		shows that this bijection also preserves lines $x\sim y$,
		so that we have an isomorphism of graphs.
	\end{proof}

	\begin{lemma}
		\label{lem boundary dbl}
		If $\cal G$ is a Fischer space
		containing no isolated points,
		then ${\cal G^\pm}/\cal G$
		is isomorphic to $\cal G$.
	\end{lemma}
	\begin{proof}
		Let $x^-,y^-\in{\cal G^\pm}\smallsetminus\cal G^+$ be arbitrary.
		Then $x^-\sim y^-$ if and only if $x\sim y$ by definition,
		and if so, then $x^-\wedge y^-=(x\wedge y)^{--} = (x\wedge y)^+\in\cal G^+$.
		Furthermore since $\cal G$ contains no isolated points,
		every $x^-\in\cal G^-$ is connected to at least one other point $y^-\in\cal G^-$.
		Therefore the point set of $\cal X = \cal G^\pm/\cal G^+$
		is $\cal G^-$,
		and $\cal X$ has lines $\{x^-,y^-\}$
		exactly when $\{x,y,x\wedge y\}$ is a line in $\cal G$.
		Thus the incidence relations of the points are the same
		(although note that the lines are not the same,
		as they have differing cardinalities).
	\end{proof}

	We now give results for specific graphs.
	\begin{lemma}
		\label{lem eigval gn}
		We record the eigenvalues, where superscripts indicate multiplicities,
		\begin{equation}
			\begin{gathered}
				\Spec(\ad(K_n)) = \{(n-1)^1,-1^{n-1}\},\quad
				k^{\cal A_{n+1}}_{\cal A_n} = n-1,\\
				\Spec(\ad(\cal A_1)) = \{0^1\},\quad
				\Spec(\ad(\cal A_2)) = \{2^1,-1^2\},\\
				\Spec(\ad(\cal A_{n\geq3})) = \{(2n-2)^1,(n-3)^n,-2^{(n+1)(n-2)/2}\}, \\
				\Spec(\ad(\cal D_{n\geq4})) = \{(4n-8)^1,(2n-8)^{n-1},-4^{n(n-3)/2},0^{(n-1)n/2}\}.
			\end{gathered}
		\end{equation}
	\end{lemma}
	\begin{proof}
		These facts are folklore;
		we used unpublished work of Hall and Spectorov for details.
		For $\cal D_n$, we can also deduce the values
		using Lemma~\ref{lem hat an is dn+1} from those for $\cal A_n$.
	\end{proof}

	For application to vertex algebras,
    we need to calculate central charges.
    Suppose that $A = M_\al(\cal G)_\FF$ is a Matsuo algebra and that $c\in\FF$.
    Then, for $x,y\in\cal G$, by \cite{matsuo}
    \begin{equation}
        \label{eq form}
        (x,y) = \begin{cases}
            2c & \text{ if } x = y, \\
            0 & \text{ if } x \not\sim y, \\
            c\al & \text{ if } x\sim y
        \end{cases}
    \end{equation}
    defines a bilinear form on $A$.
    The {\em central charge $\cc(e)$} of an idempotent $e\in A$ is $\frac{1}{2}(e,e)$.
    This matches the scaling of the form and the definition of central charge in Theorem~\ref{thm dlmn} and \cite{miyamoto}.

    Fix embeddings $\cal A_0\subseteq\cal A_1\subseteq\dotsm\subseteq\cal A_n$
    and set, in $M_\al^c(\cal A_n^\pm)$,
    \begin{equation}
        \id_i = \id_{\cal A_i},\quad
        \hat\id_i = \id_{\cal A_i^\pm},\quad
        e_i = \id_i-\id_{i-1},\quad
        \hat e_i = \hat\id_i - \id_i.
    \end{equation}
	\begin{theorem}
	  \label{thm idempots an}
	  In $A = M^c_\al(\cal A_n^\pm)$,
	  for $4\leq i<n$, for 
	  \begin{equation}
	      \eta_\al(i) = \frac{\al (i+1)}{2+2\al(i-1)}, \quad
          \hat\eta_\al(i) = \frac{\al i}{1+\al(i-1)},
        \end{equation}
      we have
	  \begin{gather}
			\begin{aligned}
				\Spec(e_i) = \{ 1, 0,
					\eta_\al(i),
					& 1 - \eta_\al(i-1),
					\eta_\al(i) - \eta_\al(i-1),\\
					&\hat\eta_\al(i) - \eta_\al(i-1),
					\hat\eta_\al(i) - \hat\eta_\al(i-1) \},
			\end{aligned}\\
  	  \Spec(\hat e_n) = \{ 1, 0,
    	   1 - \eta_\al(i-1),
    	   1 - \hat\eta_\al(i-1) \}, \\
      \cc^c_\al(e_i) = \frac{c}{2}\frac{i(2+\al(i-3))}{(1+\al(i-1))(1+\al(i-2))}, \\
      \cc^c_\al(\hat e_i) = \frac{c}{2}\frac{i(i+1)}{(1+2\al(i+1))(1+\al(i+1))}.
    \end{gather}
  \end{theorem}
  \begin{proof}
		It follows from from Lemma~\ref{lem id eigval},
		and substitutions from Lemma~\ref{lem eigval gn},
		that the eigenvalues of $\id_{\cal A_i}$ in $A$ are
		\begin{equation}
			\begin{gathered}
				\Spec(\id_{\cal A_0}) = \{ 0 \},\quad
				\Spec(\id_{\cal A_{i=1,2}}) = \{ 1, 0, \eta_\al(i) \},\\
				\Spec(\id_{\cal A_{i\geq3}}) = \{ 1, 0, \eta_\al(i), \hat\eta_\al(i) \}.
			\end{gathered}
		\end{equation}
		By observations on inclusions of eigenspaces
		and the fact that, for commuting matrices $x,y$,
		$\Spec(x-y) = \Spec(x)-\Spec(y)$,
		we deduce the spectrum of $e_i$ and $\hat e_i$.
		Namely, denote $A^{\id_{\cal A_i}}_{\phi_\al(i)}$ by $A^i_\phi$;
		then $A^{i-1}_1\subseteq A^i_1$ is clear,
		$A^i_0\subseteq A^{i-1}_0$ implies that an eigenvalue $0-\phi$ is only realised for $\phi = 0$,
		and $A^i_{\hat\eta}\subseteq A^{i-1}_{\eta,\hat\eta}$.
  \end{proof}

  In view of Theorem~\ref{thm dlmn}, we calculate the specialisation of Theorem~\ref{thm idempots an} for $\al=\frac{1}{4},c=\frac{1}{2}$ in Lemma~\ref{lem hrs an 1/4}.
  With respect to the lattice vertex algebra of $\sqrt2\rt A_n$, we find $e_i$ induces a Virasoro algebra of central charge $c_i$,
  and $\hat e_i$ is the conformal vector of a $W$-algebra of central charge $\frac{2i}{i+3}$.
  The notation $h^i_{r,s}$ indicates the highest weights of the Virasoro algebra at central charge $c_i$, as per \cite{miyamoto} or \cite{lwhwyamada}.
  \begin{lemma}
    \label{lem hrs an 1/4}
		The specialisation for $\al = \sfrac{1}{4}$, $c = \sfrac{1}{2}$
		of Theorem~\ref{thm idempots an} is,
    \begin{equation}
      \cc^{1/2}_{1/4}(e_i) = 1 - \frac{6}{(i+2)(i+3)} = c_i, \quad
      \cc^{1/2}_{1/4}(\hat e_i) = \frac{2i}{i+3}.
    \end{equation}
    \begin{align}
      0 & = h^i_{1,1}, \\
      \eta_{1/4}(i) = \frac{1}{2}\frac{i+1}{i+3}
				& = \frac{1}{2}h^i_{1,3}, \\
      1-\eta_{1/4}(i-1) = \frac{1}{2}\frac{i+4}{i+2}
				& = \frac{1}{2}h^i_{3,1}, \\
      \eta_{1/4}(i)-\eta_{1/4}(i-1) = \frac{1}{(i+2)(i+3)}
				& = \frac{1}{2}h^i_{3,3}, \\
      \hat\eta_{1/4}(i)-\eta_{1/4}(i-1) = \frac{1}{2}\frac{i(i-1)}{(i+2)(i+3)}
				& = \frac{1}{2}h^i_{3,5}, \\
      \hat\eta_{1/4}(i)-\hat\eta_{1/4}(i-1) = \frac{3}{(i+2)(i+3)}
				& = \frac{1}{2}h^i_{5,5}.
    \end{align}
  \end{lemma}
  \begin{proof}
    By direct evaluation, we see
		\begin{equation}
			\eta_{1/4}(i) = \frac{1}{2}\frac{i+1}{i+3},\quad
			\hat\eta_{1/4}(i) = \frac{i}{i+3},
		\end{equation}
		and the results are then straightforward manipulations.
  \end{proof}
  
\section{Involutions and $\cal D_n$}
  \label{sec invols}
  
  	\begin{lemma}
		\label{lem sim not sim}
		Suppose that $\cal H\subseteq\cal G$
		satisfies the hypotheses of Lemma~\ref{lem units seress},
		and that $x,y\in\cal G$ are collinear.
		If $x,y\in\cal H^{\not\sim}$, then $x\wedge y\in\cal H^{\not\sim}$.
	\end{lemma}
	\begin{proof}
		Suppose that $x,y\in\cal H^{\not\sim}$.
		Then $x,y$ are $0$-eigenvectors for $\id_{\cal H}$ in $A = M_\al(\cal G)_{\FF(\al)}$,
		$\al$ an indeterminate;
		our calculations will take place in $A$.
		Since $\id_\cal H$ is Seress,
		$xy$ is again a $0$-eigenvector of $\id_{\cal H}$.
		As $xy = \frac{\al}{2}(x + y - x\wedge y)$,
		$x\wedge y$ must also be a $0$-eigenvector.
		The $0$-eigenvectors of $\id_\cal H$ are classified in $\cal G'$
		for any $\cal G'\subseteq\cal G$ such that $\cal H\subseteq\cal G'$ is very regular,
		by Lemma~\ref{lem units 10eigvect},
		so that either $x\wedge y\in\cal H^{\not\sim}$
		or $x\wedge y\in\cal H^\sim$ and there exists
		$\cal H\subseteq\cal G'\ni x\wedge y$.
		In this latter case,
		the only $0$-eigenvector of $\id_\cal H$ in the span of $\cal G'$
		has full support in $\cal G'$,
		so that $\cal G' = \cal H\cup\{x\wedge y\}$,
		contradicting that $\cal G'$ is connected.
		Therefore $x\wedge y\in\cal H^{\not\sim}$.
	\end{proof}

	\begin{lemma}
		\label{lem id dn z2}
		The fusion rules of $\id_{\cal D_i}$ and
		\begin{equation}
			f_i = \id_{\cal D_i} - \id_{\cal D_{i-1}}
		\end{equation}
		in $M_\al(\cal D_m)$, $3\leq i < m$,
		are $\ZZ/2$-graded.
	\end{lemma}
	\begin{proof}
		The eigenvectors of $x = \id_{\cal D_i}$ are
		$\{1,0,\eta_\al(i),\eta'_\al(i)\}$.
		We will show that $\Phi_+\cup\Phi_0- = \{1,0,\eta_\al(i)\}\cup\{\eta'_\al(i)\}$
		is a $\ZZ/2$-graded partition of the fusion rules.
		We first observe that the $\eta'_\al(i)$-eigenvectors
		are of the form $x^+ - x^-$
		for $x\in\cal A_i^\sim\subseteq\cal A_m$ using the identification $\cal D_m = \cal A_m^\pm$
		from Lemma~\ref{lem hat an is dn+1}.
		We can verify by direct computation that $\id_{\cal D_i}(x^+-x^-) = \eta'_\al(i)(x^+-x^-)$.
		Furthermore note that the quotient graph
		of $\cal D_m$ by $\{ x^+-x^-\mid x\in\cal A_{m-1}\}$
		is exactly $\cal A_{m-1}^\pm / \cal A_{m-1} \cong \cal A_{m-1}$
		(see Lemma~\ref{lem boundary dbl}),
		and the image of $\id_{\cal D_i}$ under this map
		is a scalar multiple of $\id_{\cal A_{i-1}}$.
		Every vector which is annihilated in the quotient is a $\eta'_\al(i)$-eigenvector,
		so in particular no $\eta_\al(i)$-eigenvector
		is mapped to $0$.
		As $\id_{\cal A_{i-1}}$ has only $3$ distinct eigenvalues in $M_\al(\cal A_{m-1})$
		by Lemma~\ref{lem id eigval},
		and the image of $1,0$-eigenvectors are again $1,0$-eigenvectors,
		it follows that the $\eta_\al(i)$-eigenspace of $\id_{\cal D_i}$
		is mapped to the $\eta_\al(i-1)$-eigenspace of $\id_{\cal A_{i-1}}$
		and the $\eta'_\al(i)$-eigenspace is completely annihilated,
		so that all $\eta'_\al(i)$-eigenvectors lie in the span of
		$\{ x^+ - x^-\mid x\in\cal A_i^\sim \}$.

		Let $t = \tau(\id_{\cal D_i})$ be the map
		\begin{equation}
			x\mapsto\begin{cases}
				x^\ep & \text{ if }x\in\cal A_i\cup\cal A_i^{\not\sim},\\
				x^{-\ep} & \text{ if }x\in\cal A_i^\sim.\\
			\end{cases}
		\end{equation}
		Observe that $t$ inverts the $\eta'_\al(i)$-eigenspace of $\id_{\cal D_i}$
		and fixes the other eigenspaces.
		By showing that $t$ is an automorphism of $A = M_\al(\cal G)$,
		together with Lemma~\ref{lem units seress},
		we show that the fusion rules of $\id_{\cal D_i}$
		are a subset of Table~\ref{tbl-di-fus},
		which is $\ZZ/2$-graded.
		\begin{table}[ht]
		\begin{center}
		\renewcommand{\arraystretch}{2.0}
		\setlength{\tabcolsep}{0.5em}
		\begin{tabular}{c|cccc}
				\;\;$\star$\;\; & $1$ & $0$ & $\eta_\al(i)$ & $\eta'_\al(i)$ \\
			\hline
				$1$ & $\{1\}$ & $\emptyset$ & $\{\eta_\al(i)\}$ & $\{\eta'_\al(i)\}$ \\
				$0$ &	& $\{0\}$ & $\{\eta_\al(i)\}$ & $\{\eta'_\al(i)\}$ \\
				$\eta_\al(i)$ &	&	& $\{1,0,\eta_\al(i)\}$ & $\{\eta'_\al(i)\}$ \\
				$\eta'_\al(i)$ &	&	&	& $\{1, 0, \eta_\al(i)\}$ \\
		\end{tabular}
		\caption{Fusion rules $\Phi$ of $\id_{\cal D_i}$}
			\label{tbl-di-fus}
		\end{center}
		\end{table}

		Again identify $\cal D_m$ as $\cal A_{m-1}^\pm$.
		Let $\ep,\eta\in\{+,-\}$ and $x,y\in\cal A_{m-1}\subseteq\cal A_{m-1}^\pm$.
		We will consider the product $\wedge$ on collinear points $x^\ep,y^\eta$
		from the subspaces $\cal D_i,\cal D_i^\sim$ and $\cal D_i^{\not\sim}$.

		If $x^\ep,y^\eta\in\cal D_i$
		then $x^\ep\wedge y^\eta\in\cal D_i$,
		since $\cal D_i$ is a closed subspace.
		If $x^\ep,y^\eta\in\cal D_i^{\not\sim}$
		then $x^\ep\wedge y^\eta\in\cal D_i^{\not\sim}$
		by Lemma~\ref{lem sim not sim}.
		If $x^\ep\in\cal D_i^{\sim},y^\eta\in\cal D_i$
		then $x^\ep\wedge y^\eta\in\cal D_i^\sim$,
		as $y\sim(x\wedge y)$ rules out $x^\ep\wedge y^\eta\in\cal D_i^{\not\sim}$
		and $x^\ep\wedge y^\eta\in\cal D_i$ would force $x^\ep\in\cal D_i$,
		a contradiction.
		If $x^\ep\in\cal D_i^{\sim},y^\eta\in\cal D_i^{\not\sim}$
		then $x^\ep\wedge y^\eta\in\cal D_i^\sim$,
		as $y\sim(x\wedge y)$ rules out $x^\ep\wedge y^\eta\in\cal D_i$
		and $x^\ep\wedge y^\eta\in\cal D_i$ would force $x^\ep\in\cal D_i^{\not\sim}$,
		a contradiction.

		Finally, suppose that $x^\ep,y^\eta\in\cal D_i^\sim$.
		We show that $x^\ep\wedge y^\eta\in\cal D_i\cup\cal D_i^{\not\sim}$.
		It is sufficient to show that for $x,y\in\cal A_{i-1}^\sim$ in $\cal A_m$
		we have $x\wedge y\in\cal A_{i-1}\cup\cal A_{i-1}^{\not\sim}$.
		Suppose that the points of $\cal A_{i-1}$
		are labelled by transpositions in $\Sym(i)$ with support $\{1,\dotsc,i\}$
		inside $\Sym(m+1)$ with support $\{1,\dotsc,m+1\}$.
		Then $x,y$ are labelled $(i_x,j_x),(i_y,j_y)$ respectively
		with $i_x,i_y\in\{1,\dotsc,i\}$ and $j_x,j_y\in\{i+1,\dotsc,m+1\}$.
		That $x\sim y$ implies that either $i_x = i_y$ or $j_x = j_y$.
		Thus $x\wedge y$ is labelled $(j_x,j_y)$ or $(i_x,i_y)$ respectively,
		and hence $x\wedge y\in\cal A_{i-1}\cup\cal A_{i-1}^{\not\sim}$.

		To show that $t$ is an automorphism of $M_\al(\cal G)$,
		by linearity it suffices to show that for any $x^\ep,y^\eta\in\cal G$
		we have
		\begin{equation}
			\label{eq t aut?}
			(x^\ep)^t(y^\eta)^t = (x^\ep y^\eta)^t.
		\end{equation}
		When $x\not\sim y$, both sides are seen to be $0$.
		By a case-by-case analysis for $x^\ep,y^\eta$ coming from the subspaces $\cal D_i,\cal D_i^\sim$ and $\cal D_i^{\not\sim}$,
		using our information on $\wedge$ calculated previously,
		we see that \eqref{eq t aut?} is satisfied in all cases,
		for example, when $x^\ep,y^\eta\in\cal D_i^\sim$,
		\begin{equation}
			\begin{aligned}
				(x^\ep)^t(y^\eta)^t & = x^{-\ep}y^{-\eta} = \frac{\al}{2}(x^{-\ep} + y^{-\eta} - x^{-\ep}\wedge y^{-\eta}),\\
				(x^\ep y^\eta)^t & = \frac{\al}{2}(x^{\ep} + y^{\eta} - x^{\ep}\wedge y^{\eta})^t = \frac{\al}{2}(x^{-\ep} + y^{-\eta} - x^{\ep}\wedge y^{\eta}),
			\end{aligned}
		\end{equation}
		and as $x^{-\ep}\wedge y^{-\eta} = x^\ep\wedge y^\eta$,
		we have the desired equality.

		Therefore $t$ is an automorphism,
		and is the Miyamoto involution of $\id_{\cal D_i}$.
	\end{proof}

  \begin{theorem}
    \label{thm invols dn}
    Let $\tau_i$ be the Miyamoto involution $\tau(\id_{\cal D_i})\in\Aut(M_\al(\cal D_m))$
    of $\id_{\cal D_n}$ for some embedding $\cal D_i\subseteq\cal D_m$, $i\geq3$.
    Then $\tau_i$ has an action on the Fischer space $\cal D_m$ and on $W(\rt D_m)/\Z(W(\rt D_m))$,
    $\tau_i,\tau_j$ are conjugate in $G = \Aut(W(\rt D_m)/\Z(W(\rt D_m)))\subseteq\Aut(M_\al(\cal D_m))$
    if and only if $i=j$,
    and in particular $\tau_i$ is not the inner automorphism of a transposition in $W(\rt D_m)$.
    If $t$ is $G$-conjugate to $\tau_i$,
    then there exists an embedding $\cal D_i\subseteq\cal D_m$
    such that $t = \tau(\id_{\cal D_i})$.
  \end{theorem}
  \begin{proof}
  	    It follows from the proof of Lemma~\ref{lem id dn z2}
		that $\tau_i$ acts by swapping points in $\cal D_m$
		which are not collinear.
		On the other hand, for any $x\in\cal D_m$
		we know that $\tau(x)$ acts on $\cal D_m$ by permuting collinear points (see \eqref{eq miyamatsuo}).
		Therefore $\tau(\id_{\cal D_i})$
		is not in the conjugacy class of any Miyamoto involution $\tau(x)$ for $x\in\cal D_m$,
		which are the transpositions in $W(\rt D_m)$.
		It also follows that $\tau_i$ acts as $-1$ on a subspace of dimension $\size{A_{i-1}^\sim}$.
		As $\size{A_{i-1}^\sim}\neq\size{A_{j-1}^\sim}$ for $i\neq j$,
		and conjugation preserves the dimensions of eigenspaces,
		we have that $\tau_i,\tau_j$ cannot be conjugate for $i\neq j$.

        Recall that, if $t$ is an automorphism of an algebra $A$
        and $e\in A$ is a $\Phi$-axis for some $\Phi$,
        then $e^t$ is again a $\Phi$-axis.
        Furthermore, when $\Phi$ is $\ZZ/2$-graded and $\tau(e)$ is the Miyamoto involution of $e$,
        we have $\tau(e^t) = \tau(e)^t$.
		Therefore $\tau(x^{\tau(\id_{\cal D_i})}) = \tau(x)^{\tau(\id_{\cal D_i})}$,
		so that the action of $\tau(\id_{\cal D_i})$ on $\cal D_m$
		induces an action on $\{\tau(x)\mid x\in\cal D_m\}$ and the subgroup of $\Aut(M_\al(\cal D_m))$ it generates.
		By \cite{asch3} and \cite{hrs}, the Miyamoto involutions of $M_\al(\cal D_m)$,
		corresponding to involutions of points in the Fischer space $\cal D_m$,
		generate $W(\rt D_m)/\Z(W(\rt D_m))$.
        
        Suppose that $t = \tau_i^g$ for some $g\in G$.
        Let $P$ be the set of points not fixed by $\tau_i$ on $\cal D_m$.
        The embedding $\cal D_i\subseteq\cal D_m$ such that $\tau_i = \tau(\id_{\cal D_i})$ is the unique embedding of $\cal D_i\subseteq\cal D_m$
        such that, if $\cal D_m = \cal A_{m-1}^\pm$
        and $\cal A_{i-1} = \cal D_i\cap\cal A_{m-1}$,
        then $P = (A_{i-1}^\pm)^\sim$.
        Thus we can recover $\cal D_i\subseteq\cal D_m$ and $\id_i$ from the set $P$.
        Now any element $g\in G$ has an action on the transpositions of $G$,
        which we have identified with points of the Fischer space $\cal D_m$.
        Hence $t = \tau_i^g$ acts on $\cal D_m$ by fixing all points except $P^g$.
        The points $P^g$ uniquely identify an embedding $\cal D_i^g\subseteq\cal D_m$ and hence $\id_{\cal D_i^g} = \id_i^g$,
        so that $t = \tau(\id_i^g)$ as required.
  \end{proof}


\begin{thebibliography}{WWWW}

\renewcommand*{\thefootnote}{\fnsymbol{footnote}}

	\bibitem[A97]{asch3}
		M.~Aschbacher,
		{\em $3$-Transposition Groups},
		CUP 1997.

	\bibitem[C05]{carter}
		R.~Carter,
		{\em Lie Algebras of Finite and Affine Type},
		CUP 2005.
	
	\bibitem[CR15]{alonso}
		A.~Castillo-Ramirez,
		{\em Associative Subalgebras of Low-Dimensional Majorana Algebras},
		J. Alg. 421: 159--188, 2015.
		{\tt arXiv:1310.0285}
		
	\bibitem[DMR15]{jordan3trgps}
	    T. De Medts, F. Rehren,
	    {\em Jordan algebras and $3$-transposition groups},
	    16pp, submitted.
	    {\tt arXiv:1502.05657}

	\bibitem[DLMN96]{dlmn}
		C.~Dong, H.~Li, G.~Mason, S.~P.~Norton,
		{\em Associative subalgebras of the Griess algebra and related topics},
		in {\em The Monster and Lie algebras}
		(proceedings, ed. J.~Ferrar, K.~Harada),
		Ohio State / de Gruyter, 1998.
		{\tt arXiv:q-alg/9607013}

	\bibitem[GAP]{gap}
		The GAP~Group,
		{\em GAP -- Groups, Algorithms, and Programming},
		Version 4.7.5, 2014.

	\bibitem[GR01]{agt}
		C. Godsil, G. Royle,
		{\em Algebraic Graph Theory},
		Graduate Texts in Mathematics, Springer, 2001.

	\bibitem[H93]{hall-genth3}
		J.~I.~Hall,
		{\em The general theory of 3-transposition groups},
		Math. Proc. Camb. Phil. Soc. 114: 269--294, 1993.

	\bibitem[HRS14]{hrs}
		J.~I.~Hall, F.~Rehren, S.~Shpectorov,
		{\em Primitive axial algebras of Jordan type},
		J. Alg. 437: 79--115, 2015.
		{\tt arXiv:1403.1898}

	\bibitem[MAGMA]{magma}
		W. Bosma, J. Cannon, C. Playoust,
		{\em The Magma algebra system. I. The user language},
		J. Symbolic Comput. 24: 235--265, 1997.

	\bibitem[M03]{matsuo}
		A.~Matsuo,
		{\em $3$-Transposition Groups of Symplectic Type and Vertex Operator Algebras},
		preprint%
		\footnote{preferred over the published version
			in J. Math. Soc. Japan 57: 639--649, 2005
		}.
		{\tt arXiv:math/0311400}

	\bibitem[MN93]{mn}
		W. Meyer, W. Neutsch,
		{\em Associative Subalgebras of the Griess Algebra},
		J. Algebra 158: 1--17, 1993.

	\bibitem[M96]{miyamoto}
		M.~Miyamoto,
		{\em Griess algebras and conformal vectors in vertex operator algebras},
		J. Alg. 179: 523--548, 1996.
	
	\bibitem[P79]{passman}
	    D. Passman,
	    {\em The Algebraic Structure of Group Rings},
	    Dover Publications, 2011.

    \bibitem[Y01]{lwhwyamada}
        H. Yamada,
        {\em Highest weight vectors with small weights in the vertex operator algebra associated with a lattice of type $\sqrt2\rt A_l$},
        Comm. Alg. 29: 1311--1324, 2001.
\end{thebibliography}
\end{document}